\documentclass[12pt, reqno]{amsart}
\usepackage{amsmath, amsthm, amscd, amsfonts, amssymb, graphicx, color}
\usepackage[bookmarksnumbered, colorlinks, plainpages]{hyperref}

\input{mathrsfs.sty}

\newtheorem{theorem}{Theorem}[section]
\newtheorem{lemma}[theorem]{Lemma}
\newtheorem{proposition}[theorem]{Proposition}
\newtheorem{corollary}[theorem]{Corollary}
\theoremstyle{definition}

\theoremstyle{remark}
\newtheorem{remark}[theorem]{Remark}
\numberwithin{equation}{section}
\begin{document}

\title [Unitarily invariant norm inequalities]{Unitarily invariant norm inequalities involving $G_1$ operators}

\author[M. Bakherad]{ Mojtaba Bakherad}

\address{ Department of Mathematics, Faculty of Mathematics, University of Sistan and Baluchestan, Zahedan, I.R.Iran.}

\email{mojtaba.bakherad@yahoo.com; bakherad@member.ams.org}


\begin{abstract}
In this paper, we present some upper bounds for unitarily invariant norms inequalities.
Among other inequalities, we show some upper bounds for the Hilbert-Schmidt norm. In particular, we prove
\begin{align*}
\|f(A)Xg(B)\pm g(B)Xf(A)\|_2\leq \left\|\frac{(I+|A|)X(I+|B|)+(I+|B|)X(I+|A|)}{d_Ad_B}\right\|_2,
\end{align*}
where $A, B, X\in\mathbb{M}_n$ such that $A$, $B$ are Hermitian with $\sigma (A)\cup\sigma(B)\subset\mathbb{D}$ and $f, g$ are  analytic on the
complex unit disk $\mathbb{{D}}$, $g(0)=f(0)=1$,  $\textrm{Re}(f)>0$ and $\textrm{Re}(g)>0$.
\end{abstract} \maketitle
\section{Introduction}
 Let  ${\mathbb B}(\mathscr H)$ be the $C^{*}$-algebra of all bounded linear operators on a separable complex  Hilbert space ${\mathscr H}$ with the identity $I$. In the case when ${\textrm dim}{\mathscr H }=n$, we determine ${\mathbb B}({\mathscr H})$ by the matrix
algebra $\mathbb{M}_n$ of all $n\times n$ matrices having associated with entries in the complex field.
If $z\in\mathbb{C}$, then we write $z$ instead of $zI$. For any operator $A$ in the algebra ${\mathbb K}({\mathscr H})$ of all compact
operators, we denote by $\left\{s_j(A)\right\}$ the sequence of singular values of $A$, i.e.
the eigenvalues $\lambda_j(|A|)$, where $|A|=(A^*A)^{1\over2}$, enumerated as
$s_1(A)\geq s_2(A)\geq\cdots$ in decreasing order and repeated
according to multiplicity. If the rank $A$ is $n$, we put $s_k(A) = 0$ for any $k > n$. Note that $s_j(X) = s_j(X^*) = s_j(|X|)$
and $s_j(AXB)\leq \|A\|\|B\|s_j(X)$$\,(j=1, 2,\cdots)$ for all $A, B\in{\mathbb B}({\mathscr H})$ and all $X\in{\mathbb K}({\mathscr H})$.

A unitarily invariant norm is a map $|||\,\cdot\,||| :{\mathbb K}({\mathscr H})\longrightarrow[0,\infty]$ given by
$|||A||| =g(s_1(A), s_2(A), \cdots )$, where $g$ is a symmetric norming function. The set
$\mathcal{C}_{|||\,\cdot\,|||}$ including $\left\{A \in{\mathbb K}({\mathscr H}) : |||A||| < \infty\right\}$
is a closed self-adjoint ideal $\mathcal{J}$ of ${\mathbb B}({\mathscr H})$ containing finite rank operators. It enjoys the property \cite{bha}:
\begin{align}\label{kho}
|||AXB|||\leq \|A\|\|B\||||X|||
\end{align}
for $A, B\in{\mathbb B}({\mathscr H})$ and $X\in\mathcal{J}$. Inequality \eqref{kho} implies that $|||UXV|||=|||X|||$, where
$U$ and $V$ are arbitrary unitaries in ${\mathbb B}({\mathscr H})$ and $X\in\mathcal{J}$. In addition, employing the polar decomposition of
$X =W|X|$ with $W$ a partial isometry and \eqref{kho}, we have
$|||X||| = |||\, |X|\, |||$.
An operator $A\in{\mathbb K}({\mathscr H})$ is called Hilbert-Schmidt if $\|A\|_2=\left(\sum_{j=1}^{\infty}s_j^2(A)\right)^{1/2}<\infty$.
 The Hilbert-Schmidt norm is a unitarily invariant norm. For $A=[a_{ij}]\in\mathbb{M}_n$, it holds that $\|A\|_2=\Big{(}\sum_{i,j=1}^n|a_{i,j}|^2\Big{)}^{1/2}$.
We use the notation $A\oplus B$ for the diagonal block matrix $\textrm{diag}(A,B)$. Its
singular values are $s_1(A), s_1(B), s_2(A), s_2(B), \cdots$. It is evident that
$$\left|\left|\left|\left[\begin{array}{cc}
 0&A\\
 B&0
 \end{array}\right]\right|\right|\right|=\left|\left|\left|\,|A|\oplus|B|\,\right|\right|\right|=\left|\left|\left|A\oplus B\right|\right|\right|,$$
\begin{align*}
||A\oplus B|| = \max\{\|A\|,\|B\|\}\quad\textrm {and}\quad||A\oplus B||_2=\left(\|A\|_2^2+\|B\|_2^2\right)^{\frac{1}{2}}.
\end{align*}
 The inequalities involving unitarily invariant norms have been of  special interest; see e.g., \cite{bakh1,bakhmos,bakhkit, 1010} and references therein.\\

An operator $A\in{\mathbb B}({\mathscr H})$ is called $G_1$ operator if the growth condition
\begin{align*}
\|(z-A)^{-1}\|={1\over \textrm{dist}(z, \sigma(A))}
\end{align*}
holds for all $z$ not in the spectrum $\sigma(A)$ of $A$, where $\textrm{dist}(z, \sigma(A))$ denotes the
distance between $z$ and $\sigma(A)$. It is known that normal (more generally, hyponormal) operators
are $G_1$ operators (see e.g., \cite{put}).
Let $A\in{\mathbb B}({\mathscr H})$ and $f$ be a function which is analytic on an open neighborhood
$\Omega$ of $\sigma(A)$ in the complex plane. Then $f(A)$ denotes the operator defined
on ${\mathscr H}$ by the Riesz-Dunford integral as
\begin{align*}
f(A)={1\over 2\pi i}\int_Cf(z)(z-A)^{-1}dz,
\end{align*}
where $C$ is a positively oriented simple closed rectifiable contour surrounding
$\sigma(A)$ in $\Omega$ (see e.g., \cite[p. 568]{du}). The spectral mapping theorem asserts that
$\sigma(f(A)) = f(\sigma(A))$. Throughout this note, $\mathbb{D} = \left\{z \in \mathbb{C} : |z| < 1\right\}$ denotes the
unit disk, $\partial\mathbb{{D}}$ stands for the boundary of $\mathbb{{D}}$ and $d_A = \textrm{dist}(\partial\mathbb{{D}}, \sigma(A))$. In addition,
we denote
\begin{align*}
\mathfrak{A}=\left\{f:\mathbb{{D}}\rightarrow \mathbb{C}: f\,\textrm{is analytic}, \textrm{Re}(f)>0\, \textrm{and}\,f(0)=1\right\}.
\end{align*}
The Sylvester type equations $AXB\pm X = C$ have been investigated in matrix theory; see \cite{BL}.
 Several perturbation bounds for the norm of sum or difference of operators
have been presented in the literature by employing some integral representations
of certain functions; see \cite{hirz, kit2, sed} and references therein. \\

In the recent paper \cite{kit2}, Kittaneh  showed that the following inequality involving $f\in\mathfrak{A}$
\begin{align*}
|||f(A)X- Xf(B)|||\leq\frac{2}{d_Ad_B}|||AX-XB|||,
\end{align*}
where  $A, B, X\in{\mathbb B}({\mathscr H})$ such that $A$ and $B$ are $G_1$ operators with $\sigma (A)\cup\sigma(B)\subset\mathbb{D}$. In \cite{mosl}, the authors
extended this inequality for two functions $f, g\in\mathfrak{A}$ as follows
\begin{align}\label{mos-kit1}
|||f(A)X-Xg(B)|||\leq\frac{2\sqrt2}{d_Ad_B}|||\,|AX|+|XB|\,|||
\end{align}
and
\begin{align}\label{mos-kit2}
|||f(A)X+ Xg(B)|||\leq\frac{2\sqrt2}{d_Ad_B}|||\,|AXB|+|X|\,|||,
\end{align}
in which $A, B, X\in{\mathbb B}({\mathscr H})$ such that $A$ and $B$ are $G_1$ operators with $\sigma (A)\cup\sigma(B)\subset\mathbb{D}$. They also showed that
\begin{align}\label{kit3}
|||f(A)Xg(B)-X|||\leq\frac{2\sqrt2}{d_Ad_B}|||\,|AX|+|XB|\,|||
\end{align}
and
\begin{align}\label{kit444}
|||f(A)Xg(B)+X|||\leq\frac{2\sqrt2}{d_Ad_B}|||\,|AXB|+|X|\,|||,
\end{align}
where  $A, B, X\in{\mathbb B}({\mathscr H})$ such that $A$ and $B$ are $G_1$ operators with $\sigma (A)\cup\sigma(B)\subset\mathbb{D}$.

In this paper, by using some ideas from \cite{kit2, mosl} we present some upper bounds for unitarily invariant norms of the forms $|||f(A)X+ X\bar{f}(A)|||$
and $|||f(A)X- X\bar{f}(A)|||$ involving $G_1$ operator  and $f\in\mathfrak{A}$. We also present  the Hilbert-Schmidt norm inequality of the form
\begin{align*}
\|f(A)Xg(B)\pm g(B)Xf(A)\|_2\leq \left\|\frac{(I+|A|)X(I+|B|)+(I+|B|)X(I+|A|)}{d_Ad_B}\right\|_2,
\end{align*}
where  $A, B, X\in\mathbb{M}_n$ such that $A$ and $B$ are Hermitian matrices with $\sigma (A)\cup\sigma(B)\subset\mathbb{D}$ and $f, g\in\mathfrak{A}$.


\section{main results}
Our first result is some upper bounds for the Hilbert-Schmidt norm inequalities.
\begin{theorem}\label{12}
Let $A, B\in\mathbb{M}_n$ be Hermitian matrices with $\sigma (A)\cup\sigma(B)\subset\mathbb{D}$ and $f, g\in\mathfrak{A}$. Then
\begin{align*}
\|f(A)X+Xg(B)&\pm f(A)Xg(B)\|_2\\&\leq \left\|\frac{X+|A|X}{d_A}+\frac{X+X|B|}{d_B}+\frac{\left(I+|A|\right)X\left(I+|B|\right)}{d_Ad_B}\right\|_2
\end{align*}
and
\begin{align*}
\|f(A)Xg(B)\pm g(B)Xf(A)\|_2\leq \left\|\frac{(I+|A|)X(I+|B|)+(I+|B|)X(I+|A|)}{d_Ad_B}\right\|_2,
\end{align*}
where  $X \in\mathbb{M}_n$.
\end{theorem}
\begin{proof}
Let $A = U\Lambda U^*$ and $B = V\Gamma V^*$ be the spectral decomposition of $A$ and $B$ such that $\Lambda=\textrm{diag}(\lambda_1,\cdots, \lambda_n)$,
$\Gamma=\textrm{diag}(\gamma_1,\cdots, \gamma_n)$ and let $U^*XV := [y_{jk}]$. It follows from
 $|e^{i \alpha}-\lambda_j|\geq d_A$ and $|e^{i \beta}-\gamma_k|\geq d_B$ that
 \begin{align*}
&\|f(A)X+Xg(B)\pm f(A)Xg(B)\|_2^2\\&=\sum_{j,k}|f(\lambda_j)y_{j,k}+y_{j,k}g(\gamma_k)\pm f(\lambda_j)y_{j,k}g(\gamma_k)|^2
\\&=\sum_{j,k}|f(\lambda_j)\pm f(\lambda_j)g(\gamma_k)+g(\gamma_k)|^2|y_{j,k}|^2
\\&=\sum_{j,k}\left|\int_0^{2\pi}\int_0^{2\pi}\frac{e^{i\alpha}+\lambda_j}{e^{i\alpha}-\lambda_j}
+\frac{e^{i\beta}+\gamma_k}{e^{i\beta}-\gamma_k}\pm\frac{(e^{i\alpha}+\lambda_j)(e^{i\beta}+\gamma_k)}{(e^{i\alpha}-\lambda_j)(e^{i\beta}-\gamma_k)}
d\mu(\alpha)d\mu(\beta)\right|^2|y_{j,k}|^2
\\&\leq\sum_{j,k}\left(\int_0^{2\pi}\int_0^{2\pi}\frac{|e^{i\alpha}+\lambda_j|}{|e^{i\alpha}-\lambda_j|}
+\frac{|e^{i\beta}+\gamma_k|}{|e^{i\beta}-\gamma_k|}+\frac{|e^{i\alpha}+\lambda_j||e^{i\beta}+\gamma_k|}{|e^{i\alpha}-\lambda_j||e^{i\beta}-\gamma_k|}
d\mu(\alpha)d\mu(\beta)\right)^2|y_{j,k}|^2
\\&\leq\sum_{j,k}\left(\int_0^{2\pi}\int_0^{2\pi}\frac{1+|\lambda_j|}{d_A}
+\frac{(1+|\lambda_j|)(1+|\gamma_k|)}{d_Ad_B}
+\frac{1+|\gamma_k|}{d_B}d\mu(\alpha)d\mu(\beta)\right)^2|y_{j,k}|^2
\\&\leq\sum_{j,k}\left(\frac{1+|\lambda_j|}{d_A}
+\frac{1+|\gamma_k|}{d_B}+\frac{(1+|\lambda_j|)(1+|\gamma_k|)}{d_Ad_B}
\right)^2|y_{j,k}|^2
\\&=\left\|\frac{X+|A|X}{d_A}+\frac{X+X|B|}{d_B}+\frac{\left(I+|A|\right)X\left(I+|B|\right)}{d_Ad_B}\right\|_2^2.
\end{align*}
Then we get the first inequality. Similarly,
\begin{align*}
&\|f(A)Xg(B)\pm g(B)Xf(A)\|_2^2\\&=\sum_{j,k}|f(\lambda_j)y_{j,k}g(\gamma_k)\pm g(\gamma_j)y_{j,k}f(\lambda_k) |^2
\\&=\sum_{j,k}|f(\lambda_j)g(\gamma_k)\pm g(\gamma_j)f(\lambda_k)|^2|y_{j,k}|^2
\\&=\sum_{j,k}\left|\int_0^{2\pi}\int_0^{2\pi}
\frac{(e^{i\alpha}+\lambda_j)(e^{i\beta}+\gamma_k)}{(e^{i\alpha}-\lambda_j)(e^{i\beta}-\gamma_k)}\pm \frac{(e^{i\beta}+\gamma_j)(e^{i\alpha}+\lambda_k)}{(e^{i\beta}-\gamma_j)(e^{i\alpha}-\lambda_k)}
d\mu(\alpha)d\mu(\beta)\right|^2|y_{j,k}|^2
\\&\leq\sum_{j,k}\left(\int_0^{2\pi}\int_0^{2\pi}
\frac{|e^{i\alpha}+\lambda_j||e^{i\beta}+\gamma_k|}{|e^{i\alpha}-\lambda_j||e^{i\beta}-\gamma_k|}+ \frac{|e^{i\beta}+\gamma_j||e^{i\alpha}+\lambda_k|}{|e^{i\beta}-\gamma_j||e^{i\alpha}-\lambda_k|}
d\mu(\alpha)d\mu(\beta)\right)^2|y_{j,k}|^2
\\&\leq\sum_{j,k}\left(\int_0^{2\pi}\int_0^{2\pi}
\frac{(1+|\lambda_j|)(1+|\gamma_k|)}{d_Ad_B}+ \frac{(1+|\gamma_j|)(1+|\lambda_k|)}{d_Ad_B}
d\mu(\alpha)d\mu(\beta)\right)^2|y_{j,k}|^2
\\&\leq\sum_{j,k}\left(
\frac{(1+|\lambda_j|)(1+|\gamma_k|)}{d_Ad_B}+ \frac{(1+|\gamma_j|)(1+|\lambda_k|)}{d_Ad_B}
\right)^2|y_{j,k}|^2\\&\leq\sum_{j,k}\left(
\frac{(1+|\lambda_j|)y_{j,k}(1+|\gamma_k|)}{d_Ad_B}+ \frac{(1+|\gamma_j|)y_{j,k}(1+|\lambda_k|)}{d_Ad_B}
\right)^2
\\&=\left\|\frac{(I+|A|)X(I+|B|)+(I+|B|)X(I+|A|)}{d_Ad_B}\right\|_2.
\end{align*}
\end{proof}
Now, if we put $X=I$ in Theorem \ref{12}, then we get the next result.
\begin{corollary}\label{123}
Let $A, B\in\mathbb{M}_n$ be Hermitian matrices with $\sigma (A)\cup\sigma(B)\subset\mathbb{D}$ and $f, g\in\mathfrak{A}$. Then
\begin{align*}
\|f(A)+g(B)\pm f(A)g(B)\|_2\leq \left\|\frac{I+|A|}{d_A}+\frac{I+|B|}{d_B}+\frac{\left(I+|A|\right)\left(I+|B|\right)}{d_Ad_B}\right\|_2
\end{align*}
and
\begin{align*}
\|f(A)g(B)\pm g(B)f(A)\|_2\leq \left\|\frac{(I+|A|)(I+|B|)+(I+|B|)(I+|A|)}{d_Ad_B}\right\|_2.
\end{align*}
\end{corollary}

To prove the next results, the following lemma is required.
\begin{lemma}\label{man}
Let $A, B, X, Y\in{\mathbb B}({\mathscr H})$ such that $X$ and $Y$ are compact. Then\\

 $(a)\,\,s_j(AX\pm YB)\leq 2\sqrt{\|A\|\|B\|}s_j(X\oplus Y)\,\,(j=1,2,\cdots);$
\\

 $(b)\,\,|||(AX\pm YB)\oplus0|||\leq 2\sqrt{\|A\|\|B\|}\,|||X\oplus Y|||$.
\\
\end{lemma}
\begin{proof}
Using  \cite[Theorem 2.2]{hirz} we have
\begin{align*}
s_j(AX\pm YB)\leq (\|A\|+\|B\|)s_j(X\oplus Y)\,\,(j=1,2,\cdots).
\end{align*}
If we replace $A$, $B$, $X$ and $Y$ by $tA$, $\frac{B}{t}$, $\frac{X}{t}$ and $tY$, respectively, then we get
\begin{align*}
s_j(AX\pm YB)\leq (t\|A\|+\frac{\|B\|}{t})s_j(X\oplus Y)\,\,(j=1,2,\cdots).
\end{align*}
It follows from $\min_{t>0}(t\|A\|+\frac{\|B\|}{t})=2\sqrt{\|A\|\|B\|}$ that we reach the first inequality. The second inequality
can be proven by the first inequality and the Ky Fan dominance theorem \cite[Theorme IV.2.2]{bha}; see also \cite{asad}.
\end{proof}
Now, by applying Lemma \ref{man} we obtain the following result.
\begin{theorem}\label{+1}
Let $A, B, X, Y\in{\mathbb B}({\mathscr H})$  and $f, g\in\mathfrak{A}$. Then
{\footnotesize\begin{align*}
\left|\left|\left|\big((f(A)-g(B))X\pm Y(f(B)-g(A))\big)\oplus0\right|\right|\right|\leq {4\sqrt{2}\over d_Ad_B}\,\||A|+|B|\|\left|\left|\left|X\oplus Y\right|\right|\right|
\end{align*}}
and
{\footnotesize\begin{align*}
\left|\left|\left|\big((f(A)+g(B))X\pm Y(f(B)+g(A))\big)\oplus0\right|\right|\right|\leq {4\sqrt{2}\over d_Ad_B}\,\|I+|AB|\|\left|\left|\left|X\oplus Y\right|\right|\right|,
\end{align*}}
where  $X, Y$ are  compact and $A,B$ are $G_1$ operators with $\sigma (A)\cup\sigma(B)\subset\mathbb{D}$.
\end{theorem}
\begin{proof}
Using Lemma \ref{man} and inequalities \eqref{mos-kit1} and \eqref{mos-kit2} we have
\begin{align*}
\big|\big|&\big|\big((f(A)-g(B))X\pm Y(f(B)-g(A))\big)\oplus0\big|\big|\big|\\&
\leq 2\|f(A)-g(B)\|^{1\over2}\|f(B)-g(A)\|^{1\over2}|||X\oplus Y|||
\,\,\,(\textrm{by Lemma \ref{man})}\\&
\leq 2\,\sqrt{{2\sqrt{2}\over d_Ad_B}\||A|+|B|\|}\sqrt{{2\sqrt{2}\over d_Ad_B}\||B|+|A|\|}\left|\left|\left|X\oplus Y\right|\right|\right|
\\&\qquad\qquad\qquad\qquad\qquad\qquad\qquad\qquad(\textrm{by inequality \eqref{mos-kit1})}\\&
={4\sqrt{2}\over d_Ad_B}\,\||A|+|B|\|\left|\left|\left|X\oplus Y\right|\right|\right|.
\end{align*}
Similarly,
\begin{align*}
\big|\big|&\big|\big((f(A)+g(B))X\pm Y(f(B)+g(A))\big)\oplus0\big|\big|\big|\\&
\leq 2\|f(A)+g(B)\|^{1\over2}\|f(B)+g(A)\|^{1\over2}|||X\oplus Y|||\,\,\,(\textrm{by Lemma \ref{man})}\\&
\leq 2\,\sqrt{{2\sqrt{2}\over d_Ad_B}\|I+|AB|\|}\sqrt{{2\sqrt{2}\over d_Ad_B}\|I+|AB|\|}\left|\left|\left|X\oplus Y\right|\right|\right|
\\&\qquad\qquad\qquad\qquad\qquad\qquad\qquad\qquad(\textrm{by inequalities \eqref{mos-kit2}})\\&
={4\sqrt{2}\over d_Ad_B}\,\|I+|AB|\|\left|\left|\left|X\oplus Y\right|\right|\right|.
\end{align*}
\end{proof}

\begin{theorem}\label{man1}
Let $A, B\in{\mathbb B}({\mathscr H})$ be $G_1$ operators with $\sigma (A)\cup\sigma (B)\subset\mathbb{D}$ and $f\in\mathfrak{A}$. Then for every $X \in{\mathbb B}({\mathscr H})$

\begin{align}\label{main2}
\left|\left|\left|f(A)X+X\bar{f}(B)\right|\right|\right|\leq{2\over d_Ad_B}\left|\left|\left|X-AXB^*\right|\right|\right|
\end{align}
and
\begin{align}\label{main444}
\left|\left|\left|f(A)X-X\bar{f}(B)\right|\right|\right|\leq{2\sqrt{2}\over d_Ad_B}\left|\left|\left|\,|AX|+|XB^\ast|\,\right|\right|\right|,
\end{align}
\end{theorem}
\begin{proof}
Using the Herglotz representation theorem (see e.g., \cite[p. 21]{do}) we have
\begin{align*}
f(z)=\int_0^{2\pi}{e^{i\alpha}+z\over e^{i\alpha}-z}d\mu(\alpha)+i\textrm{Im}\,f(0)=\int_0^{2\pi}{e^{i\alpha}+z\over e^{i\alpha}-z}d\mu(\alpha),
\end{align*}
where $\mu$ is a positive Borel measure on the interval $[0,2\pi]$ with finite total mass $\int_0^{2\pi}d\mu(\alpha)=f(0)=1$. Hence
\begin{align*}
\bar{f}({z})=\overline{{\int_0^{2\pi}{e^{i\alpha}+{z}\over e^{i\alpha}-{z}}d\mu(\alpha)}}=\int_0^{2\pi}{e^{-i\alpha}+\bar{z}\over e^{-i\alpha}-\bar{z}}d\mu(\alpha),
\end{align*}
where $\bar{f}$ is the conjugate function of $f$ (i.e.,  $\bar{f}f=|f|^2$). So
\begin{align*}
f(A)X+& X\bar{f}(B)\\&=\int_0^{2\pi}\left(e^{i\alpha}+A\right) \left(e^{i\alpha}-A\right)^{-1}X+ X\left(e^{-i\alpha}+B^*\right) \left(e^{-i\alpha}-B^\ast\right)^{-1}d\mu(\alpha)\\&=\int_0^{2\pi}\left(e^{i\alpha}-A\right)^{-1}\Big[ \left(e^{i\alpha}+A\right)X\left(e^{-i\alpha}-B^*\right)
\\&\qquad\quad+\left(e^{i\alpha}-A\right)X\left(e^{-i\alpha}+B^\ast\right)\Big] \left(e^{-i\alpha}-B^\ast\right)^{-1}d\mu(\alpha)\\&=2\int_0^{2\pi}\left(e^{i\alpha}-A\right)^{-1}(X-AXB^\ast)\left(e^{-i\alpha}-B^\ast\right)^{-1}d\mu(\alpha).
\end{align*}
Hence
\begin{align*}
&\left|\left|\left|f(A)X+ X\bar{f}(B)\right|\right|\right|\\&=\left|\left|\left|\int_0^{2\pi}\left(e^{i\alpha}+A\right) \left(e^{i\alpha}-A\right)^{-1}X+ X\left(e^{-i\alpha}+B^\ast\right) \left(e^{-i\alpha}-B^\ast\right)^{-1}d\mu(\alpha)\right|\right|\right|
\\&=2\,\left|\left|\left|\int_0^{2\pi}\left(e^{i\alpha}-A\right)^{-1}(X-AXB^\ast)\left(e^{-i\alpha}-B^\ast\right)^{-1}d\mu(\alpha)\right|\right|\right|
\\&\leq2\int_0^{2\pi}\left|\left|\left|\left(e^{i\alpha}-A\right)^{-1}(X-AXB^\ast)\left(e^{-i\alpha}-B^\ast\right)^{-1}\right|\right|\right|d\mu(\alpha)
\\&\leq2\int_0^{2\pi}\|\left(e^{i\alpha}-A\right)^{-1}\|\|(e^{i\alpha}-B)^{-1}\|\left|\left|\left|X-AXB^\ast\right|\right|\right|d\mu(\alpha)
\\&\qquad\qquad\qquad\qquad\qquad\qquad (\textrm{by inequality}\,\eqref{kho}).
\end{align*}
Since $A$ and $B$  are $G_1$ operators, it follows from $\left\|\left(e^{i\alpha}-A\right)^{-1}\right\|={1\over \textrm{dist}(e^{i\alpha}, \sigma(A))}\leq{1\over \textrm{dist}(\partial \mathbb{D}, \sigma(A))}={1\over d_A}$
and $\left\|\left(e^{i\alpha}-B\right)^{-1}\right\|\leq{1\over d_B}$
that
 \begin{align*}
\left|\left|\left|f(A)X+X\bar{f}(B)\right|\right|\right|&\leq\left({2\over d_Ad_B}\int_0^{2\pi}d\mu(\alpha)\right)\left|\left|\left|X-AXB^\ast\right|\right|\right|
\\&=\left({2\over d_Ad_B}f(0)\right)\left|\left|\left|X-AXB^\ast\right|\right|\right|\\&={2\over d_Ad_B}\left|\left|\left|X-AXB^\ast\right|\right|\right|.
\end{align*}
Then we have the first inequality. Using the inequality
\begin{align*}
|||e^{-i\alpha}AX+e^{i\alpha}XB^\ast|||&=\left|\left|\left|\left[\begin{array}{cc}
 e^{-i\alpha}AX+e^{i\alpha}XB^\ast&0\\
 0&0
 \end{array}\right]\right|\right|\right|\\&
 =\left|\left|\left|\left[\begin{array}{cc}
 e^{-i\alpha}&e^{i\alpha}\\
 0&0
 \end{array}\right]\left[\begin{array}{cc}
 AX&0\\
 XB^\ast&0
 \end{array}\right]\right|\right|\right|\\&\leq
 \left\|\left[\begin{array}{cc}
 e^{-i\alpha}&e^{i\alpha}\\
 0&0
 \end{array}\right]\right\|\left|\left|\left|\left[\begin{array}{cc}
 AX&0\\
 XB^\ast&0
 \end{array}\right]\right|\right|\right|(\textrm{by inequality}\,\eqref{kho})
 \\&=
 \sqrt{2}\,\left|\left|\left|\,\left|\left[\begin{array}{cc}
 AX&0\\
 XB^\ast&0
 \end{array}\right]\right|\,\right|\right|\right|\\&
 =\sqrt{2}\left|\left|\left|(|AX|^2+|XB^\ast|^2)^{1\over2}\oplus0\right|\right|\right|\\&
 \leq \sqrt{2}\left|\left|\left|(|AX|+|XB^\ast|)\oplus0\right|\right|\right|\\&\qquad\qquad(\textrm{applying \cite[p. 775]{ando123} to the function} \,h(t)=t^{1\over2})
 \end{align*}
 the Ky Fan dominance theorem we have
 \begin{align}\label{dadar}
|||e^{-i\beta}AX+e^{i\alpha}XB^\ast||| \leq\sqrt{2} \left|\left|\left|\,|AX|+|XB^\ast|\,\right|\right|\right|.
 \end{align}
 It follows from \eqref{dadar} and the same argument of the proof of the first inequality that we have the second inequality and this completes the proof.
\end{proof}
\begin{remark}
Let $f(x+yi)=u(x,y)+v(x,y)i$, where $u, v$ are real and imaginary parts of $f$, respectively. If $f,\bar{f}\in\mathfrak{A}$, then the Cauchy-Riemann equations for complex analytic functions (i.e., $\frac{\partial u}{\partial x}=\frac{\partial v}{\partial y}$ and $\frac{\partial u}{\partial y}=-\frac{\partial v}{\partial x}$) implies that $v(x,y)=k$ for some $k\in\mathcal{C}$. The condition $f(0)=1$ conclude that $v(x,y)=0$. Hence,  $f$ is a real valued function. So, for arbitrary functions $f, g\in\mathfrak{A}$,  we can not replace $g$ by $\bar{f}$ in inequalities \eqref{mos-kit1} and  \eqref{mos-kit2}.  Thus, in Theorem \ref{man1} we
have been established some upper bounds for $|||f(A)X+X\bar{f}(B)|||$ and $|||f(A)X-X\bar{f}(B)|||$ in terms of $\left|\left|\left|X-AXB^\ast\right|\right|\right|$ and $\left|\left|\left|\,|AX|+|XB^\ast|\,\right|\right|\right|$, respectively,  that can not be derived from inequality \eqref{mos-kit1} and  \eqref{mos-kit2} for an arbitrary function $f\in\mathfrak{A}$.
 \end{remark}
 \begin{remark}
If  $A, B\in{\mathbb B}({\mathscr H})$ are $G_1$ operators with $\sigma (A)\cup\sigma (B)\subset\mathbb{D}$ and $f\in\mathfrak{A}$, then with a similar argument in the proof of Theorem \ref{man1} we get the following inequalities
\begin{align}\label{man222}
\left|\left|\left|\bar{f}(A)X+X{f}(B)\right|\right|\right|\leq{2\over d_Ad_B}\left|\left|\left|X-A^\ast XB\right|\right|\right|
\end{align}
and
\begin{align*}
\left|\left|\left|\bar {f}(A)X-X{f}(B)\right|\right|\right|\leq{2\sqrt{2}\over d_Ad_B}\left|\left|\left|\,|A^\ast X|+|XB|\,\right|\right|\right|,
\end{align*}
where  $X \in{\mathbb B}({\mathscr H})$.
\end{remark}
\begin{remark}
For an arbitrary operator $A\in{\mathbb B}({\mathscr H})$, the numerical range is definition by $W(A)=\{\langle Ax, x\rangle: x\in {\mathscr H},\| x \|=1\}$. It is well-known that $W(A)$ is a bounded convex subset of the complex plane $\mathbb{C}$. Its closure $\overline{W(A)}$ contains $\sigma(A)$ and is contained in $\left\{z \in \mathbb{C} : |z| \leq \|A\|\right\}$. In \cite{hil}, it is shown
\begin{align*}
\frac{1}{\textrm{dist}(z, \sigma(A))}\leq\left\|(z-A)^{-1}\right\|\qquad(z\not\in\sigma(A))
\end{align*}
and
\begin{align*}
\left\Vert(z-A)^{-1}\right\Vert\leq\frac{1}{\textrm{dist}(z,\overline{W(A)})}\qquad(z\not\in \overline{W(A)}).
\end{align*}
Now, if we replace the hypophysis $G_1$ operators by the conditions $\overline{W(A)}\cup \overline{W(B)}\subseteq \mathbb{D}$ in Theorem \ref{man1}, then  in inequalities \eqref{mos-kit1}-\eqref{kit444}, the constants $d_A$ and $d_B$ interchange to $D_A$ and $D_B$, respectively,  where
$D_A={ \textrm{dist}(\partial \mathbb{D}, \overline{W(A)})}$, $D_B={ \textrm{dist}(\partial \mathbb{D}, \overline{W(A)})}$. Also inequalities \eqref{main2} and \eqref{main444} appear of the forms
\begin{align*}
\left|\left|\left|f(A)X+X\bar{f}(B)\right|\right|\right|\leq{2\over D_AD_B}\left|\left|\left|X-AXB^*\right|\right|\right|
\end{align*}
and
\begin{align*}
\left|\left|\left|f(A)X-X\bar{f}(B)\right|\right|\right|\leq{2\sqrt{2}\over D_AD_B}\left|\left|\left|\,|AX|+|XB^\ast|\,\right|\right|\right|.
\end{align*}
where $f\in\mathfrak{A}$. For example, for every contraction operator $A$ (i.e., $A^*A\leq I$) and $0< \epsilon <1$, the operator $\epsilon A$ has the property $\overline{W(\epsilon A)}\subseteq \mathbb{D}$.
\end{remark}
If we take $X=I$ in Theorem \ref{man1}, then we get the following result.
\begin{corollary}
Let $A, B\in{\mathbb B}({\mathscr H})$ be normal operators with $\sigma (A)\cup\sigma (B)\subset\mathbb{D}$ and $f\in\mathfrak{A}$. Then for every $X \in{\mathbb B}({\mathscr H})$
\begin{align*}
\left|\left|\left|f(A)+\bar{f}(B)\right|\right|\right|\leq{2\over d_Ad_B}\left|\left|\left|I-AB^\ast\right|\right|\right|.
\end{align*}
In particular, for $B=A$ we have
\begin{align*}
\left|\left|\left|\textrm{Re}(f(A))\right|\right|\right|\leq{1\over d^2_A}\left|\left|\left|I-AA^\ast\right|\right|\right|.
\end{align*}
\end{corollary}
For the next result we need the following lemma (see also \cite{yam=kit}).
\begin{lemma}\label{mo=kit}
If $A, B, X\in{\mathbb B}({\mathscr H})$ such that $A$ and $B$ are self-adjoint and $0<mI\leq X$ for some positive real number $m$, then
\begin{align*}
m\left|\left|\left|A-B\right|\right|\right|\leq\left|\left|\left|AX+XB\right|\right|\right|.
\end{align*}
\end{lemma}
\begin{proof}
\begin{align*}
m\left|\left|\left|A-B\right|\right|\right|&\leq\frac{1}{2}\left|\left|\left|(A-B)X+X(A-B)\right|\right|\right|\qquad(\textrm {by \cite[Lemma 3.1]{van}})\\&
=\frac{1}{2}\left|\left|\left|AX-XB+(XA-BX)\right|\right|\right|\\&\leq
\frac{1}{2}\left(\left|\left|\left|AX-XB\right|\right|\right|+\left|\left|\left|XA-BX\right|\right|\right|\right)\\&
=\left|\left|\left|AX-XB\right|\right|\right|\qquad(\textrm {since\,} \|A\|=\|A^*\|).
\end{align*}
\end{proof}
\begin{proposition}
Let $A, B\in{\mathbb B}({\mathscr H})$  be $G_1$ operators with $\sigma (A)\cup\sigma (B)\subset\mathbb{D}$, let $X\in{\mathbb B}({\mathscr H})$ such that $0<mI\leq X$ for some positive real number $m$ and $f\in\mathfrak{A}$. Then
\begin{align}\label{hashemi}
m\left|\left|\left|\textrm{Re}(f(A))-\textrm{Re}(f(B))\right|\right|\right|\leq\frac{1}{d_Ad_B}\left(\left|\left|\left|X-AXB^\ast\right|\right|\right|+\left|\left|\left|X-A^\ast XB\right|\right|\right|\right),
\end{align}
In particular, if $A$ and $B$ are unitary operators, then
\begin{align*}
m\left|\left|\left|\textrm{Re}(f(A))-\textrm{Re}(f(B))\right|\right|\right|\leq\frac{2}{d_Ad_B}\left|\left|\left|X-AXB^\ast\right|\right|\right|
\end{align*}
\end{proposition}
\begin{proof}
\begin{align*}
m\left|\left|\left|\textrm{Re}(f(A))-\textrm{Re}(f(B))\right|\right|\right|&\leq
\left|\left|\left|\textrm{Re}(f(A))X+X\textrm{Re}(f(B))\right|\right|\right|\\&
\qquad\qquad\qquad\qquad\qquad\qquad\qquad\qquad(\textrm {by Lemma \ref{mo=kit}})\\&=
\frac{1}{2}\left|\left|\left|f(A)X+X\bar{f}(B)+\bar{f}(A)X+Xf(B)\right|\right|\right|\\&\leq
\frac{1}{2}\left(\left|\left|\left|f(A)X+X\bar{f}(B)\right|\right|\right|+\left|\left|\left|\bar{f}(A)X+Xf(B)\right|\right|\right|\right)\\&\leq
\frac{1}{d_Ad_B}\left(\left|\left|\left|X-AXB^\ast\right|\right|\right|+\left|\left|\left|X-A^\ast XB\right|\right|\right|\right)\\&
\qquad\qquad\qquad\qquad\qquad(\textrm {by  inequalities \eqref{main2} and \eqref{man222}}).
\end{align*}
Hence we get the first inequality. Especially, it follows from inequality \eqref{hashemi} and equation%
 \begin{align*}
\left|\left|\left|X-AXB^\ast\right|\right|\right|=\left|\left|\left|A(A^\ast XB-X)B^\ast\right|\right|\right|=\left|\left|\left|A^\ast XB-X\right|\right|\right|=\left|\left|\left|X-A^\ast XB\right|\right|\right|.
\end{align*}%
\end{proof}

\begin{remark}
Using Lemma \ref{man} we have
\begin{align*}
\big|\big|\big|((f(A)+\bar{f}(B))X&-Y(f(B)+\bar{f}(A)))\oplus0\big|\big|\big|\\&
\leq 2\|f(A)+\bar{f}(B)\|^{1\over2}\|f(B)+\bar{f}(A)\|^{1\over2}|||X\oplus Y|||\\&
=2\|f(A)+\bar{f}(B)\||||X\oplus Y|||.
\end{align*}
Now, If we apply  inequality \eqref{main2}, then we reach
\begin{align*}
\|f(A)+\bar{f}(B)\||||X\oplus Y|||
\leq{2\over d_Ad_B}\|I-AB^\ast\|\left|\left|\left|X\oplus Y\right|\right|\right|,
\end{align*}
whence
\begin{align*}
\left|\left|\left|\big((f(A)+\bar{f}(B))X-Y(f(B)+\bar{f}(A))\big)\oplus0\right|\right|\right|
\leq \,{{4\over d_Ad_B}\|I-AB^\ast\|}\left|\left|\left|X\oplus Y\right|\right|\right|.
\end{align*}
Hence, if we put $B=A$, then we get
\begin{align*}
\left|\left|\left|\textrm{Re}(f(A))X- Y\textrm{Re}(f(A))\oplus0\right|\right|\right|\leq {2\over d^2_A}\,\|I-AA^\ast\|\|\left|\left|\left|X\oplus Y\right|\right|\right|.
\end{align*}
\end{remark}

\bigskip
\bibliographystyle{amsplain}

\end{document}